\documentclass[12pt,reqno]{article}

\usepackage[usenames]{color}
\usepackage{amssymb}
\usepackage{amsmath}
\usepackage{amsthm}
\usepackage{amsfonts}
\usepackage{amscd}
\usepackage{graphicx}

\usepackage[colorlinks=true, citecolor=blue, linktocpage]{hyperref}

\usepackage{color}
\usepackage{fullpage}
\usepackage{float}
\usepackage{url}
\usepackage{numprint}
\usepackage{enumerate}
\usepackage{soul}

\allowdisplaybreaks
\definecolor{shadecolor}{gray}{0.90}				
\def\boitegrise#1#2{\begin{centerline}{\fcolorbox{black}{shadecolor}{~
    \begin{minipage}[t]{#2}{\vphantom{~}#1\vphantom{$A_{\displaystyle{A_A}}$}}
            \end{minipage}~}}\end{centerline}\medskip}

\newcommand{\seqnum}[1]{\href{https://oeis.org/#1}{\rm \underline{#1}}}

\newcommand{\pmodd}[1]{\!\!\!\!\pmod{#1}}
\newcommand{\FibSum}{\mathcal{F}(k)}
\newcommand{\LucSum}{\mathcal{L}(k)}
\newcommand{\GibSum}{\mathcal{G}_{G_0, G_1}\!(k)}
\newcommand{\GibSumPrime}{\mathcal{G}_{G_0^\prime, G_1^\prime}\!(k)}
\newcommand{\Gisano}{\pi_{G_0,G_1}\!(m)}
\newcommand{\GisanoP}{\pi_{G_0,G_1}\!(p)}
\newcommand{\GibSeq}{\left(G_n\right)_{n \geq 0}}
\newcommand{\GibDelta}{\delta_{G_0,G_1}}
\newcommand{\Dfunc}{D_{G_n,G_{n+1}}}
\newcommand{\Dinitial}{D_{G_0,G_1}}
\newcommand{\GibSumSquared}{\mathcal{G}_{G_0, G_1}^2\!(k)}
\newcommand{\LucSumSquared}{\mathcal{L}^2(k)}
\newcommand{\FibSumSquared}{\mathcal{F}^2(k)}
\newcommand{\red}[1]{\textcolor{red}{#1}}
\newcommand{\blue}[1]{\textcolor{blue}{#1}}
\DeclareMathOperator{\lcm}{lcm}

\begin{document}


\theoremstyle{plain}
\newtheorem{theorem}{Theorem}[section]
\newtheorem{corollary}[theorem]{Corollary}
\newtheorem{lemma}[theorem]{Lemma}
\newtheorem{proposition}[theorem]{Proposition}
\newtheorem{conjecture}[theorem]{Conjecture}

\theoremstyle{definition}
\newtheorem{definition}[theorem]{Definition}
\newtheorem{example}[theorem]{Example}
\newtheorem{question}[theorem]{Question}
\newtheorem{convention}[theorem]{Convention}

\theoremstyle{remark}
\newtheorem{remark}[theorem]{Remark}

\numberwithin{equation}{section}

\begin{center}
\vskip 1cm{\LARGE\bf\textcolor{red}{
GCD of sums of $k$ consecutive Fibonacci, Lucas, and generalized Fibonacci numbers
}}
\vskip 1cm
\large
Dan Guyer and aBa Mbirika\\
Department of Mathematics\\
University of Wisconsin-Eau Claire\\
Eau Claire, WI 54702\\
USA \\
\href{mailto:guyerdm7106@uwec.edu}{\tt guyerdm7106@uwec.edu}\\
\href{mailto:mbirika@uwec.edu}{\tt mbirika@uwec.edu}
\end{center}

\vskip .2in

\begin{abstract}
We explore the sums of $k$ consecutive terms in the generalized Fibonacci sequence $\left(G_n\right)_{n \geq 0}$ given by the recurrence $G_n = G_{n-1} + G_{n-2}$ for all $n \geq 2$ with integral initial conditions $G_0$ and $G_1$. In particular, we give precise values for the greatest common divisor (GCD) of all sums of $k$ consecutive terms of $\left(G_n\right)_{n \geq 0}$. When $G_0 = 0$ and $G_1 = 1$, we yield the GCD of all sums of $k$ consecutive Fibonacci numbers, and when $G_0 = 2$ and $G_1 = 1$, we yield the GCD of all sums of $k$ consecutive Lucas numbers. Denoting the GCD of all sums of $k$ consecutive generalized Fibonacci numbers by the symbol $\mathcal{G}_{G_0, G_1}\!(k)$, we give two tantalizing characterizations for these values, one involving a simple formula in $k$ and another involving generalized Pisano periods:
\begin{align*}
    \mathcal{G}_{G_0, G_1}\!(k) &= \gcd(G_{k+1}-G_1,\, G_{k+2}-G_2)\; \mbox{and}\\
    \mathcal{G}_{G_0, G_1}\!(k) &= \mathrm{lcm}\{m \mid \pi_{G_0,G_1}\!(m) \text{ divides } k\},
\end{align*}
where $\pi_{G_0,G_1}\!(m)$ denotes the generalized Pisano period of the generalized Fibonacci sequence modulo $m$. The fact that these vastly different-looking formulas coincide leads to some surprising and delightful new understandings of the Fibonacci and Lucas numbers.
\end{abstract}

\vfill

\begin{center}
    \hrulefill
\end{center}

\noindent\textcolor{blue}{\small
\textbf{NOTE}: This version of the paper is almost identical to the version that appears in the \textit{Journal of Integer Sequences}, Vol. 24 (2021), Article 21.9.8, which you can find at
\begin{center}
\url{https://cs.uwaterloo.ca/journals/JIS/VOL24/Mbirika/mb5.html}.
\end{center}
It is different from the published version in the following two respects:
\begin{itemize}
    \item This arXiv version includes a Table of Contents.
    \item The \textit{J. Integer Seq.} version does not number theorems/lemmas/propositions/etc.~by section, whereas this arXiv version does.
\end{itemize}
Those two changes may make this arXiv version easier to navigate.
}

\newpage
\tableofcontents


\section{Introduction}\label{sec:introduction}

In the inaugural issue of the \textit{Fibonacci Quarterly} in 1963, I.~D.~Ruggles proposed the following problem in the Elementary Problems section: ``Show that the sum of twenty consecutive Fibonacci numbers is divisible by the $10^{\mathrm{th}}$ Fibonacci number $F_{10}=55$.''~\cite{Ruggles1963} Since the Ruggles problem, there have been numerous papers studying sums of consecutive Fibonacci numbers or Lucas numbers~\cite{Iyer1969, Tekcan2007, Tekcan2008, Demirturk2010, Cerin2013, Shtefan2018}. However, with all this work on consecutive sums of Fibonacci and Lucas numbers, one related topic seems to be missing from the literature, namely that of the greatest common divisor (GCD) of sums of Fibonacci and Lucas numbers. That being said, the On-Line Encyclopedia of Integer Sequences (OEIS) does have two entries, \seqnum{A210209} and \seqnum{A229339}, which give the GCDs of the sums of $k$ consecutive Fibonacci (respectively, Lucas) numbers~\cite{Sloane-OEIS}. But in those entries, no references are given to any existing papers in the literature providing rigorous proofs that confirm these sequences. More precisely, two references are given in the entry \seqnum{A210209} but they appear to have little connection to the actual sequence, and the entry \seqnum{A229339} contains no references at all. Our paper serves to fill this deficiency in the literature.

Motivated by the Ruggles problem, we observed the surprising fact that not only is the sum of any twenty consecutive Fibonacci numbers divisible by  $F_{10}$, but also that $F_{10}$ is the greatest of all the divisors of these sums. This became a main motivation for us to explore sums of any finite length of consecutive Fibonacci numbers, then for Lucas numbers, and then eventually for all possible generalized Gibonacci sequences. Appearing in the literature as early as 1901 by Tagiuri~\cite{Tagiuri1901}, the generalized Fibonacci numbers (or so-called \textit{Gibonacci numbers}\footnote{Thomas Koshy attributes Art Benjamin and Jennifer Quinn for coining this term ``Gibonacci'' in their 2003 book \textit{Proofs that Really Count: The Art of Combinatorial Proof}~\cite{Benjamin2003}.}) are defined by the recurrence
$$G_i=G_{i-1}+G_{i-2} \ \text{for all} \ i\geq 2$$
with initial conditions $G_0,G_1 \in \mathbb{Z}$. We examine the GCD of the sums of $k$ consecutive Gibonacci numbers, and consequently $k$ consecutive Fibonacci and Lucas numbers. More precisely, given $k \in \mathbb{N}$ we explore the GCD of an infinite number of finite sums
$$ \sum_{i=1}^k G_i \;,\;\; \sum_{i=2}^{k+1} G_i \;,\;\; \sum_{i=3}^{k+2} G_i \;,\;\; \ldots$$
That is, we compute the GCD of the terms in the sequence $\left( \sum_{i=0}^{k-1} G_{n+i}\right)_{n \geq 1}$. By a slight abuse of notation, we write this value as $\gcd\left\lbrace \left( \sum_{i=0}^{k-1} G_{n+i} \right)_{n \geq 1}\right\rbrace$.

\boitegrise{
\begin{convention}
For brevity, we use the symbols $\FibSum$, $\LucSum$, and $\GibSum$, respectively, to denote the three values
$$ \gcd\left\lbrace \left( \sum_{i=0}^{k-1}F_{n+i} \right)_{n \geq 1}\right\rbrace, \; \gcd\left\lbrace \left( \sum_{i=0}^{k-1}L_{n+i} \right)_{n \geq 1}\right\rbrace, \text{ and } \gcd\left\lbrace \left( \sum_{i=0}^{k-1}G_{n+i} \right)_{n \geq 1}\right\rbrace. $$
For reasons to be explained in Theorem~\ref{thm:relatively_prime_initial_values_only} and Convention~\ref{conv:use_only_relatively_prime_initial_values}, it suffices to only consider Gibonacci sequences with relatively prime initial conditions $G_0$ and $G_1$.
\vspace{-.3in}
\end{convention}}{0.9\textwidth}

\begin{remark}
Observe that when $G_0=0$ and $G_1=1$ we have $\GibSum = \FibSum$, and when $G_0=2$ and $G_1=1$ we have $\GibSum = \LucSum$. Hence in the symbols $\FibSum$ and $\LucSum$, we suppress writing the initial values since those are well known in the Fibonacci and Lucas setting.
\end{remark}

To compute $\GibSum$, we establish two very different yet equivalent characterizations for this value. One is a simple formula in $k$, namely $\GibSum = \gcd(G_{k+1}-G_1, G_{k+2}-G_2)$. Another is a formula utilizing the generalized Pisano period $\Gisano$ of the Gibonacci sequence modulo $m$, namely $\GibSum = \lcm\{m \mid \Gisano \text{ divides } k\}$. We summarize our main results in Table~\ref{table: main results}. 
\begin{table}[!h]
\renewcommand{\arraystretch}{1.7}
\begin{center}
 \begin{tabular}{|c||c|c|c|c|}
 \hline
 $k$ & \;$\FibSum$\; & \;$\LucSum$\; & $\GibSum$ & Proof in this paper \\ 
 \hline\hline
 $k\equiv 0,4,8 \pmod{12}$ & $F_{k/2}$ & $5F_{k/2}$ & $F_{k/2}^{\,\red a}$ \text{or} $5F_{k/2}^{\,\red b}$ & 
Theorem~\ref{thm: GibSum result for k cong 0,4,8}\\ 
 \hline
 $k\equiv 2,6,10 \pmod{12}$ & $L_{k/2}$ & $L_{k/2}$ & $L_{k/2}$ & 
Theorem~\ref{thm: GibSum result for k cong 2,6,10}\\
 \hline
 $k\equiv 3,9 \pmod{12}$ & 2 & 2 & $2^{\,\red c}$ & 
Theorem~\ref{thm: GibSum result for k cong 3,9}\\
 \hline
 $k\equiv 1,5,7,11 \pmod{12}$ & 1 & 1 & $1^{\,\red c}$ & 
Theorem~\ref{thm: GibSum result for k cong 1,5,7,11}\\
 \hline
\end{tabular}
\caption{Summary of our main results}
\label{table: main results}
\end{center}
\end{table}
\vspace{-.4in}
\begin{center}
\footnotesize{\red{$^a$} This value holds if and only if $\gcd(G_0 + G_2, G_1 + G_3) = 1$.\\
\red{$^b$} This value holds if and only if $\gcd(G_0 + G_2, G_1 + G_3) \neq 1$.\\
\red{$^c$}  These values hold if $G_1^2 - G_0 G_1 - G_0^2 = \pm 1$. The case when\\ $G_1^2 - G_0 G_1 - G_0^2 \neq \pm 1$ is addressed in Section~\ref{sec:interesting_applications}} 

\end{center}

The paper is broken down as follows. In Section~\ref{sec:definitions}, we give a brief overview of necessary definitions and identities; in particular, we prove a few known results whose proofs seem to be missing in the literature. In Section~\ref{sec:two_equivalent_formulas_for_GibSum}, we provide proofs of our two characterizations for $\GibSum$. In Sections~\ref{sec:main_results_even_case} and \ref{sec:main_results_odd_case}, we prove our main results for the values $\GibSum$ when $k$ is even and odd, respectively. In Section~\ref{sec:interesting_applications}, we explore three tantalizing applications of our $\GibSum$ characterizations. Finally, in Section~\ref{sec:open questions}, we provide five open questions motivated by results in this paper.


\section{Definitions and preliminary identities}\label{sec:definitions}

Many results in this section are well known, and we provide references to where a proof of each result can be found. Some other lesser ``well-known'' results have no proofs in the literature as far as we have exhaustively searched, and for those results we do provide our own proofs. We use the convention of denoting these well-known results as propositions.

\begin{definition}
The \textit{generalized Fibonacci sequence} $\left(G_n\right)_{n \geq 0}$ is defined by the recurrence relation
\begin{equation*}
G_n = G_{n-1} + G_{n-2}
\end{equation*}
for all $n \geq 2$ and with arbitrary initial conditions $G_0,G_1 \in \mathbb{Z}$. The \textit{Fibonacci sequence} $\left(F_n\right)_{n \geq 0}$ is recovered when $G_0=0$ and $G_1=1$, and the \textit{Lucas sequence} $\left(L_n\right)_{n \geq 0}$ is recovered when $G_0=2$ and $G_1=1$. For brevity, we use the term \textit{Gibonacci sequence} to refer to any generalized Fibonacci sequence.
\end{definition}

The following closed form expression for the Fibonacci sequence in Proposition~\ref{prop:Fibonacci_alpha_beta} was derived and first published by Jacques Binet in 1843, but it was known at least a century earlier by Abraham de Moivre in 1718. We include this proposition and the related Proposition~\ref{prop:Lucas_alpha_beta} that follows it because we use them to prove Identities~\eqref{eq:Koshy_1}, \eqref{eq:Koshy_2}, and \eqref{eq:Koshy_3} in Lemma~\ref{lem:Koshy_identities}. In these two propositions, we set $\alpha := \frac{1+\sqrt{5}}{2}$ and $\beta := \frac{1-\sqrt{5}}{2}$.

\begin{proposition}\label{prop:Fibonacci_alpha_beta}
For $n \in \mathbb{Z}$, the Fibonacci number $F_n$ has the closed form
$$F_n = \frac{\alpha^n - \beta^n}{\alpha - \beta}.$$
\end{proposition}

\begin{proposition}\label{prop:Lucas_alpha_beta}
For $n \in \mathbb{Z}$, the Lucas number $L_n$ has the closed form
\begin{equation*}
L_n = \alpha^n + \beta^n.
\end{equation*}
\end{proposition}

\begin{proposition}\label{prop:four_fundamental_identities}
The following five identities hold:
\begin{align}
    L_n &= F_{n+1} + F_{n-1} &\text{for all $n \in \mathbb{Z}$}\label{eq:fund_identity_1_of_4}\\
    F_{2n} &= F_n L_n &\text{for all $n \in \mathbb{Z}$}\label{eq:fund_identity_1.5_of_4}\\
    G_{m+n} &= F_{m-1}G_n+F_mG_{n+1} &\text{for all $m,n \geq 1$}\label{eq:fund_identity_2_of_4}\\
    G_i &= G_0 F_{i-1} + G_1 F_i &\text{for all $i \geq 1$}\label{eq:fund_identity_3_of_4}\\
    \sum_{i=1}^{n}G_i &= G_{n+2} - G_2 &\text{for all $n \geq 1$}\label{eq:fund_identity_4_of_4}
\end{align}
\end{proposition}

\begin{proof}
Identities~\eqref{eq:fund_identity_1_of_4}, \eqref{eq:fund_identity_1.5_of_4}, \eqref{eq:fund_identity_2_of_4}, and \eqref{eq:fund_identity_4_of_4}, respectively, are proven by Vajda in his Identities~(6), (13), (8), and (33), respectively~\cite[pp.~24,25,38]{Vajda1989}. Identity~\eqref{eq:fund_identity_3_of_4} follows from Identity~\eqref{eq:fund_identity_2_of_4} if we set $m:=i$ and $n:=0$.
\end{proof}

\begin{proposition}\label{prop:generalized_Cassini} (Generalized Cassini's Identity) 
For all $n\geq 0$, the following equality holds: $G_{n+1} G_{n-1} - G_n^2 = (-1)^n \cdot D_{G_0,G_1}$, where $\Dinitial = G_1^2 - G_0 G_1 - G_0^2$. 
\end{proposition}
\begin{proof}
A generalization of this well-known identity is stated in Rabinowitz~\cite[Theorem~8]{Rabinowitz1991}.
\end{proof}

Though many encyclopedic resources such as Vajda~\cite{Vajda1989} and Koshy~\cite{Koshy2001} give nice closed forms for $F_{j-1} + F_{j+1}$ and $L_{j-1} + L_{j+1}$, the literature surprisingly lacks a closed form for $G_{j-1} + G_{j+1}$. We fill this gap in the literature with Lemma~\ref{lem:Gib_sum_of_gapsize_2} below, and this lemma along with the three identities in Lemma~\ref{lem:Koshy_identities} helps us prove the four Gibonacci propositions to follow in Subsection~\ref{subsec:four_Gibonacci_propositions}.

\begin{lemma}\label{lem:Gib_sum_of_gapsize_2}
For all $j \geq 1$, the following identity holds:
$$ G_{j-1} + G_{j+1} = G_0 L_{j-1} + G_1 L_{j}.$$
\end{lemma}

\begin{proof}
Let $j \geq 1$ be given. Observe the sequence of equalities
\begin{align*}
   G_{j-1} + G_{j+1} &= (G_0 F_{j-2} + G_1 F_{j-1}) + (G_0 F_{j} + G_1 F_{j+1}) &\mbox{(by Proposition~\ref{prop:four_fundamental_identities}, Identity~\eqref{eq:fund_identity_3_of_4})}\\
   &= G_0(F_{j-2} + F_j) + G_1(F_{j-1} + F_{j+1})\\
   &= G_0 L_{j-1} + G_1 L_j. &\mbox{(by Proposition~\ref{prop:four_fundamental_identities}, Identity~\eqref{eq:fund_identity_1_of_4})}
\end{align*}
Hence the identity holds for all $j \geq 1$. 
\end{proof}

To prove the propositions in Subsection~\ref{subsec:four_Gibonacci_propositions}, we also utilize three identities given in Lemma~\ref{lem:Koshy_identities}. Identities~\eqref{eq:Koshy_1} and \eqref{eq:Koshy_2} of this lemma can be found in Koshy~\cite[Identities 70 and 71, p.\ 90]{Koshy2001}, but he provides no proofs. It turns out that these two identities were originally stated in 1971 (though unfortunately again without proofs) in Dudley and Tucker~\cite{Dudley1971}. The related Identity~\eqref{eq:Koshy_3}, in the form we provide and utilize in Subsection~\ref{subsec:four_Gibonacci_propositions}, does not appear to be in the literature. Due to the lack of proofs for any of these identities in the literature, for completion we prove these three identities in Lemma~\ref{lem:Koshy_identities} by proving a single identity in which these three identities hold as a consequence (see Remark~\ref{rem:infinite_family}).

\begin{lemma}\label{lem:Koshy_identities}
For all $j \geq 0$, the following three identities hold:
\begin{align}
    F_{4j+1} - 1 &= F_{2j} L_{2j+1}\label{eq:Koshy_1}\\
    F_{4j+3} - 1 &= F_{2j+2} L_{2j+1}\label{eq:Koshy_2}\\
    F_{4j+4} - 1 &= F_{2j+3} L_{2j+1}.\label{eq:Koshy_3}
\end{align}
\end{lemma}

\begin{proof}
Utilizing the closed forms for $F_n$ and $L_n$ in Propositions~\ref{prop:Fibonacci_alpha_beta} and \ref{prop:Lucas_alpha_beta}, for $r,j \in \mathbb{Z}$ we have the sequence of equalities
\begin{align*}
    F_{2j+r} L_{2j+1} &= \frac{\alpha^{2j+r} - \beta^{2j+r}}{\alpha - \beta} \cdot \left( \alpha^{2j+1} + \beta^{2j+1} \right)\\
    &= \frac{\alpha^{4j+r+1} - \beta^{4j+r+1} + \alpha^{2j+r} \beta^{2j+1} - \alpha^{2j+1} \beta^{2j+r}}{\alpha-\beta}\\
    &= \frac{\alpha^{4j+r+1} - \beta^{4j+r+1}}{\alpha - \beta} + \frac{(\alpha \beta)^{2j+1} \left( \alpha^{r-1} - \beta^{r-1} \right)}{\alpha - \beta}\\
    &= F_{4j+r+1} + (-1)^{2j+1} \cdot \frac{ \alpha^{r-1} - \beta^{r-1}}{\alpha - \beta} &\text{(since $\alpha\beta = -1$)}\\
    &= F_{4j+r+1} - F_{r-1}.
\end{align*}
If we set $r:=0$, then we have $F_{2j+0} L_{2j+1} = F_{4j+0+1} - F_{0-1}$ so Identity~\eqref{eq:Koshy_1} holds since $F_{-1} = 1$. And if we set $r:=2$, then we have $F_{2j+2} L_{2j+1} = F_{4j+2+1} - F_{2-1}$ so Identity~\eqref{eq:Koshy_2} holds since $F_{1} = 1$. Lastly if we set $r:=3$, then we have $F_{2j+3} L_{2j+1} = F_{4j+3+1} - F_{3-1}$ so Identity~\eqref{eq:Koshy_3} holds since $F_{2} = 1$.
\end{proof}

\begin{remark}\label{rem:infinite_family}
In proving Lemma~\ref{lem:Koshy_identities}, we actually proved the much stronger result that an infinite family of identities of the following form holds:
$$ F_{4j+r+1} - F_{r-1} = F_{2j+r} L_{2j+1},$$
for all $r,j \in \mathbb{Z}$. This follows since the closed formulas for $F_n$ and $L_n$, given in Propositions~\ref{prop:Fibonacci_alpha_beta} and \ref{prop:Lucas_alpha_beta}, work for all integer values of $n$.
\end{remark}


\subsection{Four Gibonacci propositions}\label{subsec:four_Gibonacci_propositions}

The following four Gibonacci identities (along with our characterizations for the values $\GibSum$ given in Subsections~\ref{subsec:gcd_characterization_of_GibSum_formula} and \ref{subsec:Pisano_characterization_of_GibSum_formula} to follow) are used in the proofs of our main results in Sections~\ref{sec:main_results_even_case} and \ref{sec:main_results_odd_case}:
\begin{align*}
    G_{4j+1} - G_1 &= F_{2j}(G_{2j} + G_{2j+2}) &&\mbox{Proposition~\ref{prop:Gib_prop_1_of_4}}\\
    G_{4j+2} - G_2 &= F_{2j}(G_{2j+1} + G_{2j+3}) &&\mbox{Proposition~\ref{prop:Gib_prop_2_of_4}}\\
    G_{4j+3} - G_1 &= L_{2j+1} G_{2j+2} &&\mbox{Proposition~\ref{prop:Gib_prop_3_of_4}}\\
    G_{4j+4} - G_2 &= L_{2j+1} G_{2j+3} &&\mbox{Proposition~\ref{prop:Gib_prop_4_of_4}}
\end{align*}
These identities are stated in Koshy but without proof~\cite[p.~214]{Koshy2001}, so for completeness we provide proofs for each proposition.

\begin{proposition}\label{prop:Gib_prop_1_of_4}
For all $j \geq 0$, the following identity holds:
$$ G_{4j+1} - G_1 = F_{2j}(G_{2j} + G_{2j+2}).$$
\end{proposition}

\begin{proof}
Let $j \geq 0$ be given. Observe the sequence of equalities
\begin{align*}
    F_{2j}(G_{2j} + G_{2j+2}) &= F_{2j} (G_0 L_{2j} + G_1 L_{2j+1}) &\mbox{(by Lemma~\ref{lem:Gib_sum_of_gapsize_2})}\\
    &= G_0 \cdot F_{2j} L_{2j} + G_1 \cdot F_{2j} L_{2j+1}\\
    &= G_0 F_{4j} + G_1 \cdot F_{2j} L_{2j+1} &\mbox{(by Proposition~\ref{prop:four_fundamental_identities}, Identity~\eqref{eq:fund_identity_1.5_of_4})}\\
    &= G_0 F_{4j} + G_1 (F_{4j+1} - 1) &\mbox{(by  Lemma~\ref{lem:Koshy_identities}, Identity~\eqref{eq:Koshy_1})}\\
    &= (G_0 F_{4j} + G_1 F_{4j+1}) - G_1\\
    &= G_{4j+1} - G_1, &\mbox{(by Proposition~\ref{prop:four_fundamental_identities}, Identity~\eqref{eq:fund_identity_3_of_4})}
\end{align*}
as desired. Hence $G_{4j+1} - G_1 = F_{2j}(G_{2j} + G_{2j+2})$ for all $j \geq 0$.
\end{proof}

\begin{proposition}\label{prop:Gib_prop_2_of_4}
For all $j \geq 0$, the following identity holds:
$$ G_{4j+2} - G_2 = F_{2j}(G_{2j+1} + G_{2j+3}).$$
\end{proposition}

\begin{proof}
Let $j \geq 0$ be given. Observe the sequence of equalities
\begin{align*}
    F_{2j}(G_{2j+1} + G_{2j+3})
    &= F_{2j} (G_0 L_{2j+1} + G_1 L_{2j+2}) &\mbox{(by Lemma~\ref{lem:Gib_sum_of_gapsize_2})}\\
    &= G_0 \cdot F_{2j} L_{2j+1} + G_1 \cdot F_{2j} L_{2j+2}\\
    &= G_0 \cdot F_{2j} L_{2j+1} + G_1 \cdot F_{2j} (L_{2j} + L_{2j+1})\\
    &= (G_0 \cdot F_{2j} L_{2j+1} + G_1 \cdot F_{2j} L_{2j+1} ) + G_1 \cdot F_{2j} L_{2j}\\
    &= (G_0 + G_1) \cdot F_{2j} L_{2j+1} + G_1 \cdot F_{2j} L_{2j}\\
    &= G_2 \cdot F_{2j} L_{2j+1} + G_1 \cdot F_{2j} L_{2j}\\
    &= G_2 \cdot(F_{4j+1} - 1) + G_1 \cdot F_{2j} L_{2j} &\hspace{-.85in}\mbox{(by Lemma~\ref{lem:Koshy_identities}, Identity~\eqref{eq:Koshy_1})}\\
    &= G_2 \cdot(F_{4j+1} - 1) + G_1 F_{4j} &\hspace{-.85in}\mbox{(by Proposition~\ref{prop:four_fundamental_identities}, Identity~\eqref{eq:fund_identity_1.5_of_4})}\\
    &= (G_1 F_{4j} + G_2 F_{4j+1}) - G_2\\
    &= G_{4j+2} - G_2,
\end{align*}
where the last equality holds since for all $i \geq 1$, the value $G_i$ can be written in the following form $G_i = G_1 F_{i-2} + G_2 F_{i-1}$ by Identity~\eqref{eq:fund_identity_2_of_4} of Proposition~\ref{prop:four_fundamental_identities},  if we set $m:=i-1$ and $n:=1$. Hence $F_{4j+2} - G_2 = F_{2j}(G_{2j+1} + G_{2j+3})$ for all $j \geq 0$.
\end{proof}

\begin{proposition}\label{prop:Gib_prop_3_of_4}
For all $j \geq 0$, the following identity holds:
$$ G_{4j+3} - G_1 = L_{2j+1} G_{2j+2}.$$
\end{proposition}

\begin{proof}
Let $j \geq 0$ be given. Observe the sequence of equalities
\begin{align*}
    G_{4j+3} - G_1 &= (G_0 F_{4j+2} + G_1 F_{4j+3}) - G_1 &\mbox{(by Proposition~\ref{prop:four_fundamental_identities}, Identity~\eqref{eq:fund_identity_3_of_4})}\\
    &= G_0 F_{4j+2} + G_1 (F_{4j+3} - 1)\\
    &= G_0 F_{4j+2} + G_1 (F_{2j+2} L_{2j+1}) &\mbox{(by Lemma~\ref{lem:Koshy_identities}, Identity~\eqref{eq:Koshy_2})}\\
    &= G_0 (F_{2j+1} L_{2j+1}) + G_1 (F_{2j+2} L_{2j+1}) &\mbox{(by Proposition~\ref{prop:four_fundamental_identities}, Identity~\eqref{eq:fund_identity_1.5_of_4})}\\
    &= L_{2j+1} (G_0 G_{2j+1} + G_1 F_{2j+2})\\
    &= L_{2j+1} G_{2j+2}, &\mbox{(by Proposition~\ref{prop:four_fundamental_identities}, Identity~\eqref{eq:fund_identity_3_of_4})}
\end{align*}
as desired. Hence $G_{4j+3} - G_1 = L_{2j+1} G_{2j+2}$ for all $j \geq 0$.
\end{proof}

\begin{proposition}\label{prop:Gib_prop_4_of_4}
For all $j \geq 0$, the following identity holds:
$$ G_{4j+4} - G_2 = L_{2j+1} G_{2j+3}.$$
\end{proposition}

\begin{proof}
Let $j \geq 0$ be given. Observe the sequence of equalities
\begin{align*}
    G_{4j+4} - G_2 &= (G_0 F_{4j+3} + G_1 F_{4j+4}) - G_2 &\mbox{(by Proposition~\ref{prop:four_fundamental_identities}, Identity~\eqref{eq:fund_identity_3_of_4})}\\
    &= (G_0 F_{4j+3} + G_1 F_{4j+4}) - (G_0 + G_1)\\
    &= G_0 (F_{4j+3} - 1) + G_1 (F_{4j+4} - 1)\\
    &= G_0 (F_{2j+2} L_{2j+1}) + G_1 (F_{4j+4} - 1) &\mbox{(by Lemma~\ref{lem:Koshy_identities}, Identity~\eqref{eq:Koshy_2})}\\
    &= G_0 (F_{2j+2} L_{2j+1}) + G_1 (F_{2j+3} L_{2j+1}) &\mbox{(by Lemma~\ref{lem:Koshy_identities}, Identity~\eqref{eq:Koshy_3})}\\
    &= L_{2j+1} (G_0 F_{2j+2} + G_1 F_{2j+3})\\
    &= L_{2j+1} G_{2j+3}, &\mbox{(by Proposition~\ref{prop:four_fundamental_identities}, Identity~\eqref{eq:fund_identity_3_of_4})}
\end{align*}
as desired. Hence $G_{4j+4} - G_2 = L_{2j+1} G_{2j+3}$ for all $j \geq 0$.
\end{proof}


\section{Two equivalent formulas used to compute \texorpdfstring{$\GibSum$}{our GCD formula}}\label{sec:two_equivalent_formulas_for_GibSum}

The first two major results of this paper are given in this section. We provide two seemingly different, yet equivalent, formulas that compute the value $\GibSum$, the GCD of the sums of $k$ consecutive Gibonacci numbers. These two different characterizations not only help prove our main results in Sections~\ref{sec:main_results_even_case} and \ref{sec:main_results_odd_case}, but also lead to some tantalizing applications in Section~\ref{sec:interesting_applications}.

\subsection{A simple GCD characterization}\label{subsec:gcd_characterization_of_GibSum_formula}

In this subection, we give our first of two characterizations for the value $\GibSum$. Moreover, we establish why it suffices to consider only the Gibonacci sequences with relatively prime initial conditions, since the value $\GibSum$ for a sequence with non-relatively prime initial values $G_0$ and $G_1$ turns out to be a multiple of the value $\GibSumPrime$ of a related sequence with relatively prime initial values $G_0^\prime$ and $G_1^\prime$.

\begin{theorem}\label{thm:Gib_Sum_Construction}
The largest integer that divides every sum of $k$ consecutive Gibonacci numbers is $\gcd(G_{k+1}-G_{1},G_{k+2}-G_{2})$. That is, $\GibSum = \gcd(G_{k+1}-G_{1},G_{k+2}-G_{2})$.
\end{theorem}

\begin{proof}
Fix $n \in \mathbb{N}$ and consider the arbitrary sum $G_n + G_{n+1} + \cdots + G_{n+(k-1)}$ of $k$ consecutive Gibonacci numbers. Then we have the sequence of equalities
\begin{align*}
    \sum_{i=0}^{k-1} G_{n+i}&= \sum_{i=1}^{n+(k-1)}G_i
- \sum_{i=1}^{n-1} G_i\\
&= (G_{(n+k-1)+2}-G_2)-(G_{(n-1)+2}-G_2)
&\mbox{(by Proposition~\ref{prop:four_fundamental_identities}, Identity~\eqref{eq:fund_identity_4_of_4})}\\
&=G_{n+k+1}-G_{n+1}\\
&=G_{(k+1)+n}-G_{n+1}\\
&=F_{n-1}G_{k+1}+F_nG_{k+2}-G_{n+1}
&\mbox{(by Proposition~\ref{prop:four_fundamental_identities}, Identity~\eqref{eq:fund_identity_2_of_4})}\\
&=F_{n-1}G_{k+1}+F_nG_{k+2}-F_nG_{2}-F_{n-1}G_{1}
&\mbox{(by Proposition~\ref{prop:four_fundamental_identities}, Identity~\eqref{eq:fund_identity_2_of_4})}\\
&=F_{n-1}G_{k+1}-F_{n-1}G_1+F_nG_{k+2}-F_nG_2\\
&=F_{n-1}(G_{k+1}-G_1)+F_n(G_{k+2}-G_2).
\end{align*}
Hence our sequence of finite sums of $k$ consecutive Gibonacci numbers can be written as
\begin{align}
    \left( \sum_{i=0}^{k-1} G_{n+i}\right)_{n \geq 1} = \Bigl(F_{n-1}(G_{k+1}-G_1)+F_n(G_{k+2}-G_2)\Bigr)_{n \geq 1}. \label{eq:Gibsum_sequence_equivalence}
\end{align}
Set $q:=\gcd(G_{k+1}-G_1,G_{k+2}-G_2)$. We will show that $q \leq \GibSum$ and that $\GibSum \leq q$, and hence $\GibSum = q$ follows. Since $q$ divides both $G_{k+1}-G_1$ and $G_{k+2}-G_2$, then
$q$ divides every term in our sequence, and therefore $q \leq \GibSum$, as desired.
Next we establish that $\GibSum \leq q$. Observe that the GCD of every term in our sequence is at most the GCD of the first two terms. Consider the GCD of the first two terms. We have the sequence of inequalities and equalities
\begin{align*}
    \GibSum
    &\leq \gcd\Big(F_{0}(G_{k+1}-G_1)+F_1(G_{k+2}-G_2) \;, \; F_{1}(G_{k+1}-G_1)+F_2(G_{k+2}-G_2)\Big)\\
    &=\gcd\big(G_{k+2}-G_2 \;,\; G_{k+1}-G_1+G_{k+2}-G_2\big)\\
    &=\gcd\big(G_{k+2}-G_2 \;,\; G_{k+1}-G_1\big)\\
    &=q,
\end{align*}
where the second equality holds by the property $\gcd(a,b + a) = \gcd(a,b)$. Thus $\GibSum \leq q$, as desired. We conclude that $\GibSum = \gcd(G_{k+1}-G_{1},G_{k+2}-G_{2})$.
\end{proof}

\begin{corollary}
The largest integer to divide every sum of $k$ consecutive Fibonacci numbers is precisely $\gcd(F_{k+1}-F_1,F_{k+2}-F_2)$.  That is, $\FibSum = \gcd(F_{k+1}-1,F_{k+2}-1)$.
\end{corollary}

\begin{corollary}
The largest integer to divide every sum of $k$ consecutive Lucas numbers is precisely $\gcd(L_{k+1}-L_{1},L_{k+2}-L_{2})$.  That is, $\LucSum =  \gcd(L_{k+1}-1,L_{k+2}-3)$.
\end{corollary}

After proving the following two results, Lemma~\ref{lem:Consec_Gib_GCD} and Theorem~\ref{thm:relatively_prime_initial_values_only}, we will conclude that it is sufficient to explore only the Gibonacci sequences which have relatively prime initial values.

\begin{lemma}\label{lem:Consec_Gib_GCD}
For all $n \in \mathbb{Z}$, the values $\gcd(G_{n+1},G_{n+2})$ and $\gcd(G_n,G_{n+1})$ coincide. In particular, $\gcd(G_0,G_1)=\gcd(G_{n},G_{n+1})$ holds for all $n\in \mathbb{Z}$
\end{lemma}
\begin{proof}
Observe the following sequence of equalities.
\begin{align*}
    \gcd(G_{n+1},G_{n+2})&=\gcd(G_{n+1},G_{n+1}+G_n)\\
    &=\gcd(G_{n+1},G_n).
\end{align*}
Hence $\gcd(G_0,G_1)=\gcd(G_n,G_{n+1})$ as desired for all $n\in \mathbb{Z}$.
\end{proof}

\begin{theorem}\label{thm:relatively_prime_initial_values_only}
Fix $G_0, G_1 \in \mathbb{Z}$ and set $d := \gcd(G_0,G_1)$. Then the GCD of every sum of $k$ consecutive Gibonacci numbers in the sequence $\GibSeq$ is $d$ times the GCD of every sum of $k$ consecutive Gibonacci numbers in the new sequence $\{G_n^\prime\}_{n=0}^\infty$ generated by the relatively prime initial conditions $G_0^\prime = \frac{G_0}{d}$ and $G_1^\prime = \frac{G_1}{d}$. In particular, we have the following:
$$ \GibSum = d \cdot \GibSumPrime.$$
\end{theorem}

\begin{proof}
Set $d := \gcd(G_0,G_1)$. By Lemma~\ref{lem:Consec_Gib_GCD}, we have $\gcd(G_{k+1},G_{k+2}) = \gcd(G_0,G_1) = d$ for all $k \in \mathbb{Z}$. By Theorem~\ref{thm:Gib_Sum_Construction}, the largest positive integer that divides every sum of $k$ consecutive Gibonacci numbers is $\gcd(G_{k+1}-G_1,G_{k+2}-G_2)$. Moreover, since $d$ divides $G_0$ and $G_1$, then $d$ divides every term in the sequence $\GibSeq$. In particular, $\frac{G_{k+1}-G_1}{d}$ and $\frac{G_{k+2}-G_2}{d}$ are integers. Observe the sequence of equalities
\begin{align*}
    \gcd(G_{k+1}-G_1,G_{k+2}-G_2)&=\gcd\left(d \cdot  \frac{G_{k+1}-G_1}{d},d \cdot \frac{G_{k+2}-G_2}{d}\right)\\
    &=d \cdot \gcd\left(\frac{G_{k+1}-G_1}{d},\frac{G_{k+2}-G_2}{d}\right).
\end{align*}
Notice that by Theorem~\ref{thm:Gib_Sum_Construction}, the value $\gcd\left(\frac{G_{k+1}-G_1}{d},\frac{G_{k+2}-G_2}{d}\right)$ is the GCD of the sum of $k$ consecutive Gibonacci in the new sequence $\{G_n^\prime\}_{n=0}^\infty$ generated by the initial values $G_0^\prime = \frac{G_0}{d}$ and $G_1^\prime = \frac{G_1}{d}$. Clearly $G_0^\prime$ and $G_1^\prime$ are relatively prime. In particular, we have
$$ \GibSum = d \cdot \GibSumPrime,$$
as desired.
\end{proof}

\boitegrise{
\begin{convention}\label{conv:use_only_relatively_prime_initial_values}
In order to give a complete classification of the GCD of every sum of $k$ consecutive Gibonacci numbers, as a consequence of Theorem~\ref{thm:relatively_prime_initial_values_only}, we need only to consider Gibonacci sequences with relatively prime initial values.
\vspace{-.2in}
\end{convention}}{0.9\textwidth}


\subsection{A generalized Pisano period characterization}\label{subsec:Pisano_characterization_of_GibSum_formula}

As in the setting of the Fibonacci and Lucas sequences modulo $m$, it is well known that the Gibonacci sequence modulo $m$ is also periodic. Hence it makes sense to consider the period $\Gisano$ of this sequence given in the following definition.

\begin{definition}\label{def:Gisano_period}
Let $m \geq 2$. The \textit{generalized Pisano period}, $\Gisano$, of the Gibonacci sequence $\GibSeq$ is the smallest positive integer $r$ such that
$$ G_r \equiv G_0 \pmodd{m} \;\;\text{ and }\;\; G_{r+1} \equiv G_1 \pmodd{m}.$$
In the Fibonacci (respectively, Lucas) setting we denote this period by $\pi_F(m)$ (respectively, $\pi_L(m)$).
\end{definition}

\begin{lemma}\label{lem:m_divides_sum_Pisano}
The value $m$ divides the sum of any $\Gisano$ consecutive Gibonacci numbers.  That is, $m$ divides $\sum_{i=1}^{\Gisano}G_{n+i}$ for any fixed $n\in \mathbb{Z}$.  
\end{lemma}

\begin{proof}
We need to prove $m$ divides the sum of the terms in a generalized Pisano period of any Gibonacci sequence. However, by the periodicity of generalized Pisano periods, it suffices to show that $m$ divides the sum of the terms in the particular generalized Pisano period given by $(G_1, G_2, \ldots, G_{\Gisano})$. By Identity~\eqref{eq:fund_identity_4_of_4} of Proposition~\ref{prop:four_fundamental_identities} we have 
$$\sum_{i=1}^{\Gisano}G_i=G_{\Gisano+2}-G_2.$$
However, by the definition of a generalized Pisano period, $G_{\Gisano+2} \equiv G_2 \pmod{m}$. Hence $m$ divides $G_{\Gisano+2}-G_2$ and therefore also divides $\sum_{i=1}^{\Gisano}G_i$ as desired. It follows that $m$ divides the sum of the terms in the particular generalized Pisano period $(G_1, G_2, \ldots, G_{\Gisano})$, which proves that $m$ divides the sum of any $\Gisano$ consecutive Gibonacci numbers. 
\end{proof}

\begin{remark}
It can be proven that the value $\Gisano$ in Lemma~\ref{lem:m_divides_sum_Pisano} is minimal with respect to the following property: If $s \in \mathbb{N}$ with $s < \Gisano$, then $m$ cannot divide the sum of every $s$ consecutive Gibonacci numbers.
\end{remark}

\begin{theorem}\label{thm:mfactor_iff_period_divides_k}
The value $\Gisano$ divides $k$
if and only if $m$ divides $\GibSum$. 
\end{theorem}

\begin{proof}
Let $k\in \mathbb{N}$ be fixed. Suppose $\Gisano$ divides $k$. By Lemma~\ref{lem:m_divides_sum_Pisano}, we know that $m$ divides the sum of any $\Gisano$ consecutive Gibonacci numbers. Thus $m$ divides any sum of $t\cdot \Gisano$ consecutive Gibonacci numbers for any $t\in \mathbb{N}$. From our assumption that $\pi_{G_0,G_1}$ divides $k$, it follows that $k=t_0 \cdot \Gisano$ for some $t_0\in \mathbb{N}$. Hence $m$ is a common divisor of any sum of $k$ consecutive Gibonacci numbers, which proves that $m$ divides the greatest common divisor $\GibSum$ as desired. 

Assume $m$ divides $\GibSum$. Then $m$ divides $\gcd(G_{k+2}-G_2,G_{k+1}-G_1)$. Thus $m$ divides $G_{k+2}-G_2$ and $m$ divides $G_{k+1}-G_1$. Hence $G_{k+2}\equiv G_2 \pmod{m}$ and $G_{k+1} \equiv G_1 \pmod{m}$. By the periodicity of the sequence $\GibSeq$ under a modulus, $\Gisano$ divides $k$.
\end{proof}

\begin{theorem}\label{thm:lcm_equiv_definition}
For all $k \geq 1$, we have $\GibSum=\lcm\{m \mid \Gisano \ \text{divides} \ k\}$.
\end{theorem}

\begin{proof}
For ease of notation, set $\ell(k) := \lcm\{m \mid \Gisano \ \text{divides} \ k\}$. Then it suffices to prove that $\GibSum$ divides $\ell(k)$ and that $\ell(k)$ divides $\GibSum$. Since both $\GibSum$ and $\ell(k)$ are strictly greater than $0$, we only need to show that any divisor of $\GibSum$ is a divisor of $\ell(k)$, and vice versa. Let $d_0$ be a divisor of $\GibSum$. Then by Theorem~\ref{thm:mfactor_iff_period_divides_k}, it follows that $\pi_{G_0,G_1}\!(d_0)$ divides $k$. Hence by definition of $\ell(k)$, we conclude that $d_0$ divides $\ell(k)$ as desired. Now, suppose that $d_1$ is a divisor of $\ell(k)$. Then by definition of $\ell(k)$, it must be that $\pi_{G_0,G_1}\!(d_1)$ divides $k$. Hence by Theorem~\ref{thm:mfactor_iff_period_divides_k}, we conclude that $d_1$ divides $\GibSum$ as desired. 

\end{proof}


\section{Main results for \texorpdfstring{$\GibSum$}{our GCD formula} when \texorpdfstring{$k$}{k} is even}\label{sec:main_results_even_case}

In this section, we provide our main results for the values $\GibSum$ when $k$ is even. There are two cases that we consider; namely, when $k \equiv 0, 4,\, \text{or } 8 \pmod{12}$ given in Subsection~\ref{subsec:0,4,8_even_case} and when $k \equiv 2, 6,\, \text{or } 10 \pmod{12}$ given in Subsection~\ref{subsec:2,6,10_even_case}. From Table~\ref{table: main results} in Section~\ref{sec:introduction}, we see that the second row, which corresponds to $k \equiv 2, 6, 10 \pmod{12}$, gives the same value $L_{k/2}$ regardless if we are considering $\FibSum$, $\LucSum$, or $\GibSum$; that is, no matter which initial values for the sequence $\{G_i\}_{n=0}^\infty$ are chosen, the values $\FibSum$, $\LucSum$, and $\GibSum$ coincide. However in the first row of this table when $k\equiv 0,4,8 \pmod{12}$, it turns out that the value of $\GibSum$ depends on the initial conditions $G_0$ and $G_1$, and hence the values $\FibSum$, $\LucSum$, and $\GibSum$ may differ. More precisely,  for a fixed $k$ such that $k\equiv 0,4,8 \pmod{12}$, we will see in the following subsection that the behavior of these latter three values depends on an easily computed parameter which we denote by $\GibDelta$, defined as $\GibDelta := \gcd(G_0 + G_2, G_1 + G_3)$.

\subsection{The \texorpdfstring{$k\equiv 0,4,8 \pmod{12}$}{k congruent 0,4,8 mod 12} case}\label{subsec:0,4,8_even_case}

Lemmas~\ref{lem:gcd_is_1_or_5} is used to conclude our penultimate result, Lemma~\ref{lem:gcd_same_for_all_n}, which essentially implies that the value of $\GibSum$ is determined soley by the value $k$ and the parameter $\GibDelta$.

\begin{remark}\label{rem:interesting_point_about_GibDelta}
It is worth noting that in this subsection, only our main result, Theorem~\ref{thm: GibSum result for k cong 0,4,8}, involves the value $k$. The two lemmas have no mention of the value $k$, and in fact, say something quite interesting about any Gibonacci sequence $\GibSeq$. In particular, as a consequence of Lemma~\ref{lem:gcd_same_for_all_n}, the value $\gcd(G_0 + G_2, G_1 + G_3)$, which is the parameter $\GibDelta$, equals 1 or 5, and moreover the value $\gcd(G_n + G_{n+2}, G_{n+1} + G_{n+3})$ equals $\GibDelta$ for all $n \geq 0$.
\end{remark}

\begin{lemma}\label{lem:gcd_is_1_or_5}
Fix an integer $i\geq 0$. It follows that the value $\gcd(G_i+G_{i+2},G_{i+1}+G_{i+3})$ is either $1$ or $5$.
\end{lemma}

\begin{proof}
Suppose $\gcd(G_0,G_1)=1$. Let $d$ be any divisor of $\gcd(G_i+G_{i+2},G_{i+1}+G_{i+3})$. Since $d$ divides the sums $G_i+G_{i+2}$ and $G_{i+1}+G_{i+3}$, we have the congruences
\begin{align}
    G_i+G_{i+2}&\equiv 0 \pmodd{d} \label{eq:n_congruence}\\
    G_{i+1}+G_{i+3}&\equiv 0 \pmodd{d} \label{eq:n+1_congruence}.
\end{align}
We can express the three values $G_{i+2},G_{i+1}$ and $G_{i+3}$, respectively, in terms of $G_i$ as follows:
\begin{align}
    G_{i+2}&\equiv-G_i \pmodd{d}
    &\mbox{(by Congruence~\eqref{eq:n_congruence})} \label{eq:n+2_in_terms_of_n}\\
 &\hfill \nonumber\\
    G_{i+1}&= G_{i+2}-G_i \nonumber\\
    &\equiv -G_i-G_i \pmodd{d} \nonumber
    &\mbox{(by Congruence~\eqref{eq:n+2_in_terms_of_n})} \nonumber\\
    &\equiv -2G_i \pmodd{d}\label{eq:n+1_in_terms_of_n}\\
&\hfill \nonumber\\
    G_{i+3}&=G_{i+2}+G_{i+1} \nonumber\\
    &\equiv -G_i-2G_i \pmodd{d} \nonumber
    &\mbox{(by Congruences~\eqref{eq:n+2_in_terms_of_n} and \eqref{eq:n+1_in_terms_of_n})} \nonumber\\
    &\equiv -3G_i \pmodd{d} \label{eq:n+3_in_terms_of_n}. 
\end{align}
Then by Congruences~\eqref{eq:n+1_congruence}, \eqref{eq:n+1_in_terms_of_n}, and \eqref{eq:n+3_in_terms_of_n}, we have
$$ 0 \equiv G_{i+1} + G_{i+3} \equiv -2 G_i - 3 G_i = -5G_i \pmodd{d},$$
and thus $5G_i \equiv 0 \pmod{d}$. Furthermore, observe that by adding Congruences~\eqref{eq:n_congruence} and \eqref{eq:n+1_congruence} we get that $G_{i+2}+G_{i+4}\equiv 0 \pmod{d}$. Hence $d$ is a divisor of $G_{i+1}+G_{i+3}$ and $G_{i+2}+G_{i+4}$. Therefore $d$ divides $\gcd(G_{i+1}+G_{i+3}, G_{i+2}+G_{i+4})$. Analogous to our previous work, we can express the three values $G_{i+3},G_{i+2}$ and $G_{i}$, respectively, in terms of $G_{i+1}$ to find that $d$ divides $5G_{i+1}$. Since $d$ divides both $5G_i$ and $5G_{i+1}$ and $\gcd(G_i,G_{i+1})=1$, it must be that $d$ divides $5$. Thus $d=1$ or $d=5$.
\end{proof}

\begin{remark}\label{rem:deltagib_attains}
It is worthy to note that both values 1 and 5 are attained in Lemma~\ref{lem:gcd_is_1_or_5}. For $i=0$ in the Fibonacci sequence, we have $\gcd(F_0 + F_2, F_1 + F_3) = \gcd(1,2) = 1$. Moreover, for $i=0$ in the Lucas sequence, we have $\gcd(L_0 + L_2, L_1 + L_3) = \gcd(5,5) = 5$. However, it is not yet clear that for fixed initial values $G_0$ and $G_1$, the values $\gcd(G_n+G_{n+2},G_{n+1}+G_{n+3})$ will be the same for all $n$. However, a consequence of Lemma~\ref{lem:gcd_same_for_all_n} will confirm the latter. But first we need to define what we mean for two Gibonacci sequences to be equivalent (up to shift) modulo $m$ for some $m \geq 2$.
\end{remark}

\begin{definition}\label{def:equivalent_up_to_shift}
Let $m \geq 2$. Let $G$ and $G'$ denote the Gibonacci sequences $\GibSeq$ and $\left(G'_n\right)_{n \geq 0}$, respectively, with corresponding generalized Pisano periods $\pi_{G_0,G_1}(m)$ and $\pi_{G'_0,G'_1}(m)$. We say that $G$ modulo $m$ is \textit{equivalent (up to shift)} to $G'$ modulo $m$ if the following two conditions hold:
\begin{enumerate}[(i)]
\item The values $\pi_{G_0,G_1}(m)$ and  $\pi_{G'_0,G'_1}(m)$ coincide.
\item For some $r \in \mathbb{Z}$, we have $G_{r+n} \equiv G'_n \pmod{m}$ for all $n \in \mathbb{Z}$.
\end{enumerate}
\end{definition}

\begin{remark}\label{rem:gcd(G_0 + G_2, G_1 + G_3) = 5}
It can be shown that the value $\GibDelta$ equals $5$ if and only if the Gibonacci sequence $\GibSeq$ modulo $5$ is equivalent (up to shift) to the Lucas sequence $\left(L_n\right)_{n \geq 0}$ modulo $5$.
\end{remark}

\begin{lemma}\label{lem:gcd_same_for_all_n}
The value $\GibDelta$ equals $1$ or $5$, and we have the following:
$$\gcd(G_n+G_{n+2},G_{n+1}+G_{n+3}) = \GibDelta$$
for all $n \geq 0$.
\end{lemma}

\begin{proof}
Let $H_n=G_n+G_{n+2}$. Observe that 
\begin{align*}
    H_n+H_{n+1}&=(G_n+G_{n+2})+(G_{n+1}+G_{n+3})\\
    &=(G_n+G_{n+1})+(G_{n+2}+G_{n+3})\\
    &=G_{n+2}+G_{n+4}\\
    &=H_{n+2}.
\end{align*}
Thus the sequence $\left ( H_n \right )_{n\geq 0}$ is itself a generalized Fibonacci sequence. By Lemma~\ref{lem:Consec_Gib_GCD}, we have $\gcd(H_0,H_1)=\gcd(H_n,H_{n+1})$ for all $n\geq 0$.
\end{proof}

We are now ready to prove the main theorem of this subsection. We utilize the parameter $\GibDelta$. Recall from Remark~\ref{rem:deltagib_attains} that $\GibDelta = 1$ for the Fibonacci sequence, $\GibDelta = 5$ for the Lucas sequence, and $\GibDelta = 1 \mbox{ or } 5$ for Gibonacci sequences.

\begin{theorem}\label{thm: GibSum result for k cong 0,4,8}
If $k\equiv 0,4,8\pmod {12}$, then $\gcd(G_{k+1}-G_1,G_{k+2}-G_2) = \GibDelta \cdot F_{k/2}$, where $\GibDelta = \gcd(G_0 + G_2, G_1 + G_3)$. In particular, we conclude the following:
\begin{align*}
    \FibSum &= F_{k/2}\\
    \LucSum &= 5 F_{k/2}\\
    \GibSum &= \GibDelta \cdot F_{k/2}.
\end{align*}
\end{theorem}

\begin{proof}
Assume $k\equiv 0,4,8\pmod {12}$. Then $k\equiv 0 \pmod{4}$. Thus $k=4j$ for some $j\in \mathbb{Z}$. Observe the sequence of equalities
\begin{align*}
    \gcd(G_{k+1}-&G_1,G_{k+2}-G_2)\\
    &=\gcd(G_{4j+1}-G_1,G_{4j+2}-G_2)\\
    &=\gcd(F_{2j}(G_{2j}+G_{2j+2}),F_{2j}(G_{2j+1}+G_{2j+3}))
    &\mbox{(by Propositions~\ref{prop:Gib_prop_1_of_4} and \ref{prop:Gib_prop_2_of_4})}\\
    &=F_{2j} \cdot \gcd(G_{2j}+G_{2j+2},G_{2j+1}+G_{2j+3})\\
    &=F_{k/2} \cdot \gcd(G_{2j}+G_{2j+2},G_{2j+1}+G_{2j+3}).
\end{align*}
Observe that from Lemma~\ref{lem:gcd_same_for_all_n}, we know $\gcd(G_{2j}+G_{2j+2},G_{2j+1}+G_{2j+3}) = \GibDelta$. Thus $\gcd(G_{k+1}-G_1,G_{k+2}-G_2) = \GibDelta \cdot F_{k/2}$. We conclude that if $k\equiv 0,4,8\pmod {12}$, then $\FibSum = F_{k/2}$ and $\LucSum = 5 F_{k/2}$ and $\GibSum = \GibDelta \cdot F_{k/2}$.
\end{proof}

\subsection{The \texorpdfstring{$k\equiv 2,6,10 \pmod{12}$}{k congruent 0,4,8 mod 12} case}\label{subsec:2,6,10_even_case}

Whereas the $k\equiv 0,4,8\pmod {12}$ case in Subsection~\ref{subsec:0,4,8_even_case} had variability in the value $\GibSum$ dependent on the initial values $G_0$ and $G_1$, the $k\equiv 2,6,10 \pmod {12}$ case in this subsection is more straightforward since all values $\FibSum$, $\LucSum$, and $\GibSum$ coincide, regardless of the initial values.

\begin{theorem}\label{thm: GibSum result for k cong 2,6,10}
If $k\equiv 2,6,10\pmod {12}$, then $\gcd(G_{k+1}-1,G_{k+2}-3) = L_{k/2}$. In particular, we conclude the following:
\begin{align*}
    \FibSum &= L_{k/2}\\
    \LucSum &= L_{k/2}\\
    \GibSum &= L_{k/2}.
\end{align*}
\end{theorem}

\begin{proof}
Assume $k\equiv 2,6,10\pmod {12}$. Then $k\equiv 2 \pmod{4}$. Thus $k=4j+2$ for some $j\in \mathbb{Z}$. Observe the sequence of equalities
\begin{align*}
    \gcd(G_{k+1}-&G_1,G_{k+2}-G_2)\\
    &=\gcd(G_{(4j+2)+1}-G_1,G_{(4j+2)+2}-G_2)\\
    &=\gcd(G_{4j+3}-G_1,G_{4j+4}-G_2)\\
    &=\gcd(L_{2j+1}G_{2j+2},L_{2j+1}G_{2j+3})
    &\mbox{(by Propositions~\ref{prop:Gib_prop_3_of_4} and \ref{prop:Gib_prop_4_of_4})}\\
    &=L_{2j+1} \cdot \gcd(G_{2j+2},G_{2j+3})\\
    &=L_{2j+1} &\mbox{(since $ \gcd(G_{2j+2},G_{2j+3}) = 1$)}\\
    &=L_{k/2}.
\end{align*}
Thus $\gcd(G_{k+1}-G_1,G_{k+2}-G_2)=L_{k/2}$. We conclude that if $k\equiv 2,6,10\pmod {12}$, then we have $\FibSum = \LucSum = \GibSum = L_{k/2}$.
\end{proof}


\section{Main results for \texorpdfstring{$\GibSum$}{our GCD formula} when \texorpdfstring{$k$}{k} is odd}\label{sec:main_results_odd_case}

The two main results in this section, Theorems~\ref{thm: GibSum result for k cong 3,9} and \ref{thm: GibSum result for k cong 1,5,7,11}, rely on the generalized Pisano period $\Gisano$ being even for all $m>2$. The period $\pi_F(m)$ of the Fibonacci sequence modulo $m$ being even for all $m>2$ is well known and proven in 1960 by Wall \cite{Wall1960}, and a clever short proof was given more recently by Elsenhans and Jahnel~\cite{Elsenhans2010}. Similarly, the period $\pi_L(m)$ of the Lucas sequence modulo $m$ is also even for all $m>2$; however, this well-known result seems to lack a proof in the literature, though it is stated in a number of sources. A corollary to the following lemmas will not only prove that the Fibonacci and Lucas periods are even, but also provides a sufficiency condition on the initial values $G_0$ and $G_1$ that will give the defined $\GibSum$ values we gave in Table~\ref{table: main results} in this case when $k$ is odd.

\subsection{A sufficiency criterion for when \texorpdfstring{$\Gisano$}{generalized Pisano period} is even for all \texorpdfstring{$m>2$}{m greater than two}}

\begin{lemma}\label{lem:D_function_invariance}
Let $\Dfunc$ denote the value $G_{n+1}^2 - G_n G_{n+1} - G_n^2$. Then the following holds:
\begin{align}
    \Dfunc = (-1)^n \cdot \Dinitial \label{eq:Dfunc_invariance}
\end{align}
for all $n \geq 0$. In particular, we have $|\Dfunc| = |\Dinitial|$ for all $n \geq 0$.
\end{lemma}

\begin{proof}
We prove this by induction on $n$. Identity~\eqref{eq:Dfunc_invariance} clearly holds when $n = 0$. So suppose it holds for some $k \geq 0$, and consider $D_{G_{k+1}, G_{k+2}}$. Then we have
\begin{align*}
    D_{G_{k+1}, G_{k+2}} &= G_{k+2}^2 - G_{k+1} G_{k+2} - G_{k+1}^2\\
    &= (G_k + G_{k+1})^2 - G_{k+1} (G_k + G_{k+1}) - G_{k+1}^2\\
    &= G_k^2 + 2 G_k G_{k+1} + G_{k+1}^2 - G_k G_{k+1} - G_{k+1}^2 - G_{k+1}^2\\
    &= - (G_{k+1}^2 - G_k G_{k+1} - G_k^2)\\
    &= -(-1)^k \cdot \Dinitial\\
    &= (-1)^{k+1} \cdot \Dinitial,
\end{align*}
where the fifth equality holds by the induction hypothesis. Hence Identity~\eqref{eq:Dfunc_invariance} holds for all $n \geq 0$.
\end{proof}

\begin{lemma}\label{lem:Dinitial_congruence}
For all integers $m>2$, the following congruence holds:
$$ (-1)^{\Gisano} \cdot \Dinitial \equiv \Dinitial \pmodd{m}.$$
\end{lemma}

\begin{proof}
By the generalized Cassini's identity, given in Proposition~\ref{prop:generalized_Cassini}, it follows that
$$G_{n+1} G_{n-1} - G_n^2 = (-1)^n \cdot D_{G_0,G_1}.$$ Substituting $\Gisano$ for $n$ in the latter identity, we have
$$G_{\Gisano + 1}\cdot G_{\Gisano - 1} - G_{\Gisano}^2 = (-1)^{\Gisano} \cdot D_{G_0,G_1}.$$
Since $G_{\Gisano + i} \equiv G_i \pmod{m}$ for all $i$, we get the sequence of congruences
\begin{align*}
   (-1)^{\Gisano} \cdot D_{G_0,G_1} &\equiv G_1 G_{-1} - G_0^2 \pmodd{m}\\
   &\equiv G_1 (G_1 - G_0) - G_0^2 \pmodd{m}\\
   &\equiv \Dinitial \pmodd{m},
\end{align*}
and the claim holds.
\end{proof}

\begin{corollary}\label{cor:even_parity_of_Fib_Luc_Gib}
For all $m > 2$, the Pisano periods of the Fibonacci and Lucas sequences are even. In the general setting, if $\Dinitial = \pm 1$ then $\Gisano$ is even for all $m > 2$.
\end{corollary}

\begin{proof}
By Lemma~\ref{lem:Dinitial_congruence}, we have $(-1)^{\Gisano} \cdot \Dinitial \equiv \Dinitial \pmod{m}$ for all $m>2$. We address the Fibonacci, Lucas, and Gibonacci settings in three separate cases.
\smallskip
\newline
\textbf{Case 1}: If $G_0 = 0$ and $G_1 = 1$, then $\Dinitial = 1$, and we have $(-1)^{\pi_F(m)} \equiv 1 \pmod{m}$, which implies that $\pi_F(m)$ is even for all $m>2$.
\medskip
\newline
\textbf{Case 2}: If $G_0 = 2$ and $G_1 = 1$, then $\Dinitial = -5$, and we have $(-1)^{\pi_L(m)} 5 \equiv 5 \pmod{m}$, which implies that $\pi_L(m)$ is even for all $m>2$ when $\gcd(5,m)=1$. If on the other hand $\gcd(5,m) \neq 1$, then $m = 5^s t$ for some $s,t \in \mathbb{N}$ with $\gcd(5,t) = 1$. A consequence of Theorem~2 by Wall yields $\pi_L(m) = \lcm\big(\pi_L(5^s), \pi_L(t)\big)$~\cite{Wall1960}. But since $\pi_L(5)$ divides $\pi_L(5^s)$ and $\pi_L(5) = 4$, then $4$ divides $\lcm\big(\pi_L(5^s), \pi_L(t)\big)$. Hence $\pi_L(m)$ is even for all $m > 2$.
\medskip
\newline
\textbf{Case 3}: In general, for any initial values $G_0$ and $G_1$ with $\Dinitial = \pm 1$, it follows that $(-1)^{\Gisano} \equiv 1 \pmod{m}$. Hence $\Gisano$ is even for all $m > 2$ when $\Dinitial = \pm 1$.
\end{proof}

\subsection{The \texorpdfstring{$k\equiv 3,9 \pmod{12}$}{k congruent 3 or 9} case}\label{subsec:3,9_odd_case}

\begin{lemma}\label{lem:Gibsum_odd_leq_2}
If $k$ is odd and $\Gisano$ is even for all $m>2$, then $\GibSum\leq2$.
\end{lemma}

\begin{proof}
Suppose that $\Gisano$ is even for all $m>2$ and $k$ is odd.
By Theorem~\ref{thm:lcm_equiv_definition}, we have $\GibSum=\lcm\{m \mid \Gisano \ \text{divides} \ k \}$. Since $\Gisano$ is even for all $m>2$ and $k$ is odd, $\Gisano$ cannot divide $k$ for all $m>2$. Thus $\GibSum = \lcm\{m \mid \Gisano \ \text{divides} \ k \} \leq 2$ as desired. 
\end{proof}

\begin{theorem}\label{thm: GibSum result for k cong 3,9}
If $k\equiv 3,9 \pmod{12}$ and $\Gisano$ is even for all $m>2$, then $\GibSum = 2$. In particular, we conclude the following:
\begin{align*}
    \FibSum &= 2\\
    \LucSum &= 2\\
    \GibSum &= 2 \text{ if $\Dinitial = \pm 1$}.
\end{align*}
\end{theorem}

\begin{proof}
Suppose $k\equiv 3,9 \pmod{12}$ and that $\Gisano$ is even for all $m>2$. Then we have $k \equiv 0 \pmod{3}$.
Since $\gcd(G_0,G_1)=1$, the Gibonacci sequence modulo 2 is equivalent (up to shift) to the sequence
$$1,0,1,1,0,1,1,0,1,1,0,1,1\ldots$$
Hence $\pi_{G_0,G_1}(2)=3$. By assumption, $k\equiv 0 \pmod{3}$ and thus $\pi_{G_0,G_1}(2)$ divides $k$. Therefore by Theorem~\ref{thm:lcm_equiv_definition}, we have $\GibSum=\lcm\{m \mid \Gisano \ \text{divides} \ k \}\geq 2$. However by Lemma~\ref{lem:Gibsum_odd_leq_2}, $\GibSum\leq 2$ since $k$ is odd. We conclude that if $k\equiv 3,9\pmod {12}$, then we have $\FibSum = \LucSum = 2$. Also if $\Dinitial = \pm 1$, then $\Gisano$ is even and $\GibSum = 2$.
\end{proof}

\subsection{The \texorpdfstring{$k\equiv 1,5,7,11 \pmod{12}$}{k congruent 1,5,7,11} case}\label{subsec:1,5,7,11_odd_case}

\begin{theorem}\label{thm: GibSum result for k cong 1,5,7,11}
If $k\equiv 1,5,7,11 \pmod{12}$ and $\Gisano$ is even for all $m>2$, then $\GibSum = 1$. In particular, we conclude the following:
\begin{align*}
    \FibSum &= 1\\
    \LucSum &= 1\\
    \GibSum &= 1 \text{ if $\Dinitial = \pm 1$}.
\end{align*} 
\end{theorem}

\begin{proof}
Suppose $k \equiv 1,5,7,11 \pmod{12}$ and that $\Gisano$ is even for all $m>2$. As in the proof of Theorem~\ref{thm: GibSum result for k cong 3,9}, we know the fact that $\gcd(G_0,G_1) = 1$ holds implies that $\pi_{G_0,G_1}(2) = 3$. However, since we have $k\not \equiv 0 \pmod{3}$ it cannot be that $\pi_{G_0,G_1}(2)$ divides the value $k$. Therefore by Lemma~\ref{lem:Gibsum_odd_leq_2}, it follows that $\GibSum=\lcm\{m \mid \Gisano \ \text{divides} \ k \}=1$ as desired. We conclude that if $k\equiv 1,5,7,11 \pmod {12}$, then we have $\FibSum = \LucSum = 1$. Also if $\Dinitial = \pm 1$, then $\GibSum = 1$.
\end{proof}


\section{Interesting applications}\label{sec:interesting_applications}

The Pisano period characterization of the value $\GibSum$ not only yields the GCD of all sums of $k$ consecutive Gibonacci numbers, but also leads to some interesting applications. In this section, we highlight three such applications.

\subsection{Restrictions on the factors of \texorpdfstring{$\GibSum$}{GCD of k consecutive Gibonacci numbers} when \texorpdfstring{$k$}{k} is odd}

While the Fibonacci and Lucas sequences satisfy the property that both $\pi_F(m)$ and $\pi_L(m)$ are even for all $m>2$, this is not the case in general. From this, we may exhibit odd values of $k$ for which $\GibSum$ is greater than two. We provide an example of this below and we place restrictions on the values of $m$ that can make $\Gisano$ odd later in this subsection. Using the generalized Pisano period characterization of $\GibSum$, we place restrictions on the factors of $\GibSum$.

\begin{example} In this example, we will show that $\pi_{1,4}(m)=5$ only when $m=11$. Recall that by Theorems~\ref{thm: GibSum result for k cong 3,9} and \ref{thm: GibSum result for k cong 1,5,7,11}, the value $\GibSum$ is 1 or 2 if $k$ is odd and $\Dinitial = \pm 1$, where $\Dinitial$ was defined to be $G_1^2 - G_0 G_1 - G_0^2$. Hence if $\Dinitial \neq \pm 1$, it is interesting to consider what are the possible values of $\GibSum$ when $k$ is odd. We shall consider the Gibonacci sequence with initial values $G_0 = 1$ and $G_1 = 4$, which we will call the $(1,4)$-Gibonacci sequence. Observe that in this case, we have $G_1^2 - G_0 G_1 - G_0^2 = 11 \neq \pm 1$, and hence the value $\mathcal{G}_{1,4}(k)$ when $k$ is odd is not forced to be 1 or 2, necessarily. We consider the value $\mathcal{G}_{1,4}(k)$ when $k=5$ by computing it in two different ways. First, we write out the $(1,4)$-Gibonacci sequence as follows:
$$1,4,5,9,14,23,37,60,97,157,245,402,647,849,1496, \ldots$$
Consider the first four terms of the sequence $\left( \sum_{i=0}^{k-1} G_{n+i}\right)_{n \geq 1}$ when $k=5$: 
$$(4+5+9+14+23, \ 5+9+14+23+37, \  9+14+23+37+60, 14+23+37+60+97, \ldots ),$$
or equivalently, $(55, 88, 143, 231, \ldots )$. By inspection, one may suspect that $\mathcal{G}_{1,4}(5)$ is 11. This can be affirmed by our simple GCD characterization as follows:
$$\mathcal{G}_{1,4}(5) = \gcd(G_7-G_2,G_6-G_1) = \gcd(60-5,37-4) = \gcd(55,33) = 11.$$
On the other hand, by using the generalized Pisano characterization of $\GibSum$, we know
$$\mathcal{G}_{1,4}(5)=11=\lcm\{ m \ | \ \pi_{1,4}(m) \ \text{divides} \ 5 \}.$$
By Theorem~\ref{thm:mfactor_iff_period_divides_k}, since 11 divides $\mathcal{G}_{1,4}(5)$, it must be that $\pi_{1,4}(11)$ divides 5. Clearly $\pi_{1,4}(11)\neq 1$. Hence $\pi_{1,4}(11)=5$ is forced. Furthermore, since 11 is prime, its only divisors are $1$ and $11$. Hence the only divisors of $\lcm\{ m \ | \ \pi_{1,4}(m) \ \text{divides} \ 5 \}$ can be $1$ or $11$, and again employing Theorem~\ref{thm:mfactor_iff_period_divides_k}, this implies that $\pi_{1,4}(m)$ does not divide 5 for all $m\neq 1,11$. Hence, we can conclude that for the particular Gibonacci sequence with initial values $G_0=1$ and $G_1 = 4$, we know that the only modulus value $m$ that yields $\Gisano = 5$ is the value $m=11$.
\end{example}

\begin{remark}
Observe that in the previous example the values $m=11$ and $D_{1,4}=11$ coincide. When examining the $(1,24)$-Gibonacci sequence, which yields an odd period for $m=29$, we do not have $m=D_{1,24}$. However, the value $m=29$ divides $D_{1,24}=551$.
\end{remark}

\begin{proposition}[Wall, Theorem 8]\label{prop:Wall_theorem_8} If $p$ is prime and $p\equiv 3,7,13,17 \pmod{20}$, then it follows that $\pi_{G_0,G_1}(p^e) = \pi_F(p^e)$.
\end{proposition}

Since $\pi_F(m)$ is even for all $m>2$, Proposition~\ref{prop:Wall_theorem_8} yields the immediate corollary.

\begin{corollary}\label{coro:even_Gib_periods}
If $p$ is prime and $p\equiv 3,7,13,17 \pmod{20}$, then $\GisanoP$ is even no matter the choice of (coprime) initial conditions. 
\end{corollary}

\begin{theorem}\label{thm:Gibsum_factor_restrictions}
There exists no prime $p\equiv 3,7,13,17\pmod{20}$ that can be a factor of $\GibSum$ if $k$ is odd. 
\end{theorem}

\begin{proof}
Recall that by Theorem~\ref{thm:mfactor_iff_period_divides_k}, the value $\Gisano$ divides $k$
if and only if $m$ divides $\GibSum$.  So if $p$ is a prime such that $p\equiv 3,7,13,17 \pmod{20}$, then by Corollary~\ref{coro:even_Gib_periods}, we know that $\GisanoP$ is even. However, if $k$ is odd then surely $\GisanoP$ cannot divide $k$. Hence no prime of the form $p\equiv 3,7,13,17\pmod{20}$ can be a factor of $\GibSum$ if $k$ is odd. 
\end{proof}


\subsection{Largest modulus \texorpdfstring{$m$}{m} yielding a given Pisano period value \texorpdfstring{$\pi_F(m)$}{}}

It is well known that for a given modulus $m$, the corresponding Pisano period $\pi_F(m)$ is bounded above by $6m$. This problem was proposed by Freyd in 1990 and answered by Brown in 1992~\cite{Brown1992}. Moreover, this upper bound is achieved, for instance, when $m=10$ since $\pi_F(10) = 60$. Furthermore, since the value $\Gisano$ divides $\pi_F(m)$, the value $6m$ also serves as an upper bound on any Gibonacci sequence. Hence, this question of an upper bound for any generalized Pisano period of a given modulus $m$ is answered. However, a different but related question can be considered.

\begin{question}
For a given period $k$ and a Gibonacci sequence with initial values $G_0$ and $G_1$, what is the largest modulus value $m$ such that $\Gisano = k$? We answer this question in the Fibonacci setting.\footnote{This question was explored in 2018 in the Fibonacci setting by Dishong and Renault from an algorithmic approach that allows a computer to calculate all values $m$ such that $\pi_F(m) = k$~\cite{Dishong2018}. However, we answer this question from a theoretical approach utilizing the generalized Pisano period characterization of $\GibSum$.} We invite the interested reader to explore this problem in the Lucas and Gibonacci setting.
\end{question}

The following example exhibits how this question may be approached in the Fibonacci and Lucas settings, in particular, from the generalized Pisano period characterization of $\GibSum$.

\begin{example}
Let us attempt to compute the largest modulus value $m$ that yields a Pisano period $\pi_F(m)$ equal to 60. Setting $k:=60$ in Theorem~\ref{thm: GibSum result for k cong 0,4,8}, we know $\mathcal{F}(60) = F_{30} = \numprint{832040}$. By Theorem~\ref{thm:mfactor_iff_period_divides_k}, we know the following:
$$ \pi_F(m) \text{ divides } 60 \text{ if and only if } m \text{ divides } \mathcal{F}(60).$$
Hence we can conclude that if $\pi_F(m) = 60$, then $m$ divides \numprint{832040}. So certainly, we have \numprint{832040} as a potential maximum value $m$ that makes $\pi_F(m) = 60$, but the question that remains is ``Does $\pi_F(\numprint{832040})$ indeed equal 60?'' A simple computer computation reveals that this is so. Hence the largest modulus value $m$ that yields a Pisano period $\pi_F(m)$ equal to 60 is $m = \numprint{832040}$. In Theorem~\ref{thm:max_fib_modulus}, we will prove that in general for $k \equiv 0 \pmod{4}$ that $\mathcal{F}(k)$ (or equivalently $F_{k/2}$) is the actual largest modulus that produces a Pisano period equal to $k$.
\end{example}

Before we prove our main result, Theorem~\ref{thm:max_fib_modulus}, we need the following known results on the periods $\pi_F(m)$ when $m$ is a Fibonacci or Lucas number.

\begin{lemma}\label{lem:Pisano_periods_of_Fib_numbers}
The following identities hold:
\begin{align}
    \pi_F (F_i) &= \begin{cases} 2i, &\mbox{if $i \geq 4$ and even;} \\
    4i, &\mbox{if $i \geq 5$ and odd}.\end{cases} \label{eq:Pisano_of_Fib_number}\\
    \pi_F (L_i) &= \begin{cases} 4i, &\mbox{if $i \geq 2$ and even;} \\
    2i, &\mbox{if $i \geq 3$ and odd}.\end{cases} \label{eq:Pisano_of_Luc_number}
\end{align}
In particular, it follows that $\mathrm{range}(\pi_F) = \{3\} \cup \{n \in 2\mathbb{Z} \mid n \geq 6\}$.
\end{lemma}

\begin{proof}
Identity~\eqref{eq:Pisano_of_Fib_number} was first proven in 1971 by Stanley~\cite{Stanley1971} (and independently in 1989 by Ehrlich~\cite{Ehrlich1989}, who was apparently unaware of Stanley's result). Identity~\eqref{eq:Pisano_of_Luc_number} was proven in 1976 by Stanley~\cite{Stanley1976}. Moreover, in that same paper Stanley states that the range of $\pi_F$ is all even integers greater than 4 though omits the trivial result that $\pi_F(2)=3$ and hence we have $\mathrm{range}(\pi_F) = \{3\} \cup \{n \in 2\mathbb{Z} \mid n \geq 6\}$, as desired.
\end{proof}

\begin{theorem}\label{thm:max_fib_modulus}
Let $k \geq 6$ be an even integer, and set $m_F:=\FibSum$. Then $m_F$ is the largest modulus value yielding a Fibonacci period of $k$. More precisely, $\pi_F(m_F) = k$ and for all $m > m_F$, we have $\pi_F(m) \neq k$. In particular, we have the following:
$$m_F = \begin{cases} F_{k/2} &\mbox{if } k \equiv 0 \pmodd{4}, \\
L_{k/2} &\mbox{if } k \equiv 2 \pmodd{4}. \end{cases}$$
\end{theorem}

\begin{proof}
Suppose $k \geq 6$ is an even integer. Then either $k \equiv 0 \pmod{4}$ or $k \equiv 2 \pmod{4}$.
\medskip
\newline
\textbf{Case 1}: Suppose $k \equiv 0 \pmod{4}$. Set $m_F := \FibSum$. Then by Theorem~\ref{thm: GibSum result for k cong 0,4,8}, we have $m_F = F_{k/2}$. Since $k \equiv 0 \pmod{4}$ and $k \geq 6$ is even, then $k/2 \geq 4$ is even. Thus, Identity~\eqref{eq:Pisano_of_Fib_number} of Lemma~\ref{lem:Pisano_periods_of_Fib_numbers} implies that $\pi_F(F_{k/2}) = k$. Hence we have $\pi_F(m_F) = k$. We now show that there are no larger values $m > m_F$ yielding $\pi_F(m) = k$. Recall that Theorem~\ref{thm:mfactor_iff_period_divides_k} implies
$$ \pi_F(m) \text{ divides } k \text{ if and only if } m \text{ divides } \mathcal{F}(k).$$
Hence $\mathcal{F}(k)$ is the maximum potential modulus value $m$ that could make $\pi_F(m) = k$. Since we have $\pi_F(m_F) = k$ and $m_F = \mathcal{F}(k)$, then we have achieved the maximum modulus, namely $F_{k/2}$, yielding a period of $k$ when $k \equiv 0 \pmod{4}$.
\medskip
\newline
\textbf{Case 2}: Suppose $k \equiv 2 \pmod{4}$. Set $m_F := \FibSum$. Then by Theorem~\ref{thm: GibSum result for k cong 2,6,10}, we have $m_F = L_{k/2}$. Since $k \equiv 2 \pmod{4}$ and $k \geq 6$ is even, then $k/2 \geq 3$ is odd. Thus, Identity~\eqref{eq:Pisano_of_Luc_number} of Lemma~\ref{lem:Pisano_periods_of_Fib_numbers} implies that $\pi_F(L_{k/2}) = k$. Hence we have $\pi_F(m_F) = k$. By the exact same argument given in Case 1, we know that $\mathcal{F}(k)$ is the maximum potential modulus value $m$ that could make $\pi_F(m) = k$. So we have achieved the maximum modulus, namely $L_{k/2}$, yielding a period of $k$ when $k \equiv 2 \pmod{4}$.
\end{proof}


\subsection{Computing odd-indexed Lucas numbers using \texorpdfstring{$\GibSum$}{our GCD formula} characterizations}\label{subsec:application_odd_index_Lucas_numbers}

The fact that the formulas $\gcd(G_{k+1}-G_1, G_{k+2}-G_2)$ and $\lcm\{m \mid \Gisano \text{ divides } k\}$ for $\GibSum$ coincide leads to some surprising and delightful new understandings of the Fibonacci and Lucas numbers. One such example can be garnered from looking at the the $k \equiv 2, 6, 10 \pmod{12}$ row in Table~\ref{table: main results}. Given such a $k$-value, the Lucas number $L_{k/2}$ can be computed in two new ways. One is by taking a Gibonacci sequence $\GibSeq$ with any initial relatively prime initial values $G_0$ and $G_1$, then by the first $\GibSum$ characterization, we have $L_{k/2} = \gcd(G_{k+1}-G_1, G_{k+2}-G_2)$. On the other hand, by the second $\GibSum$ characterization, we have $L_{k/2} = \lcm\{m \mid \Gisano \text{ divides } k\}$. In this subsection, we consider the first of these two ways. We have effectively established an easily computable way to find any odd-index Lucas number using any Gibonacci sequence with relatively prime initial values.

\begin{theorem}\label{thm:odd_index_Lucas_number}
Let $j$ be an odd positive integer, and suppose that $\GibSeq$ is a Gibonacci sequence with relatively prime initial values $G_0$ and $G_1$. Then the $j^{\mathrm{th}}$ Lucas number $L_j$ is given by $\mathcal{G}_{G_0, G_1}\!(2j)$. More precisely, we have
$$ L_j = \gcd\big(G_{2j+1} - G_1, \; G_{2j+2} - G_2\big).$$
\end{theorem}

\begin{proof}
This follows from Theorem~\ref{thm: GibSum result for k cong 2,6,10} if we set $k:=2j$ and observe that $k \equiv 2 \pmod{4}$ since $j$ is odd.
\end{proof}

The latter theorem is quite surprising. It leads one to ponder if such a GCD-formulation can be discovered which yields the even-index Lucas numbers. But even more intriguing is the fact that we have our second characterization of the $\GibSum$ formula. More precisely, the odd index Lucas number $L_j$ is given by $\lcm\{m \mid \Gisano \ \text{divides} \ 2j\}$. Admittedly, computing the Lucas number $L_j$ using this LCM formulation is not as easily done as it is using Theorem~\ref{thm:odd_index_Lucas_number}, due to the fact that the periods $\Gisano$ are not easily computed. In the open questions section, we ask a question regarding this formulation.


\section{Open questions}\label{sec:open questions}

There are many avenues for further research motivated from the work in this present paper. The following open problems arose from the consideration of our $\GibSum$ characterizations and other questions related to our research.

\begin{question}
By examining our two equivalent definitions of $\GibSum$, we observe that when $k\equiv 2 \pmod{4}$ it follows that $L_{k/2}=\lcm\{m \mid \Gisano \ \text{divides} \ k\}$ for every possible choice of $G_0$ and $G_1$. Is there an intuitive reason why this must be true?
\end{question}

\begin{question}
Theorem~\ref{thm:Gibsum_factor_restrictions} establishes that no prime $p\equiv 3,7,13,17\pmod{20}$ can be a factor of $\GibSum$ if $k$ is odd. Which primes of the form $p\equiv 1,9,11,19\pmod{20}$ can be factors of $\GibSum$ when $k$ is odd? Can we place further restrictions on the possible factors of $\GibSum$ when $k$ is odd?
\end{question}

\begin{question}
Observe that for the Fibonacci and Lucas sequences we have that $\pi_F(m)$ and $\pi_L(m)$ are even for all $m>2$.  
For which initial values $G_0$ and $G_1$, does there exist a number $N$ such that $\Gisano$ is even for all $m>N$?
\end{question}

\begin{question}
To prove Theorem~\ref{thm:max_fib_modulus} for the maximum modulus value $m$ that yields a given period $k$ in the Fibonacci and Lucas settings, respectively, we relied heavily on Lemma~\ref{lem:Pisano_periods_of_Fib_numbers} which gave the Fibonacci and Lucas Pisano periods for moduli of the form $F_i$ and $L_i$. Can we generalize this lemma to provide conditions on initial values $G_0$ and $G_1$ that can help us predict the precise value of $\pi_{G_0, G_1}(G_i)$ for each $i$; that is, the generalized Pisano period of the sequence $\GibSeq$ modulo the Gibonacci number $G_i$?
\end{question}

\begin{question}
Can we extend our work to sums of $k$ consecutive squares of Gibonacci numbers? That is, for a fixed $k\in \mathbb{N}$ and initial values $G_0,G_1 \in \mathbb{Z}$, what is the value of $\GibSumSquared$, which we define to be $\gcd\left\lbrace \left( \sum_{i=0}^{k-1}G_{n+i}^2 \right)_{n \geq 1}\right\rbrace$? In the Fibonacci setting, small computational data leads to the following conjectural values of $\FibSumSquared$, which are the values $\GibSumSquared$ when $G_0=0$ and $G_1=1$: 
\begin{center}
    \begin{tabular}{|c||c|}
    \hline
    $k$ & $\FibSumSquared$\\ \hline\hline
    0 & $0 = \blue{F_0 L_0}$ \\ \hline
    4 & $1 \cdot 3 = \blue{F_2 L_2}$ \\ \hline
    8 & $3 \cdot 7 = \blue{F_4 L_4}$ \\ \hline
    12 & $2^4 \cdot 3^2 = \blue{F_6 L_6}$ \\ \hline
    16 & $3 \cdot 7 \cdot 47 = \blue{F_8 L_8}$ \\ \hline
    20 & $3 \cdot 5 \cdot 11 \cdot 41  = \blue{F_{10} L_{10}}$\\\hline
    \end{tabular}
    \hspace{.1in}
    \begin{tabular}{|c||c|}
    \hline
    $k$ & $\FibSumSquared$\\ \hline\hline
    1 & 1 \\ \hline
    5 & 1 \\ \hline
    9 & 2 \\ \hline
    13 & 1 \\ \hline
    17 & 1 \\ \hline
    21 & 2\\ \hline
    \end{tabular}
    \hspace{.1in}
    \begin{tabular}{|c||c|}
    \hline
    $k$ & $\FibSumSquared$\\ \hline\hline
    2 & $1 = \blue{F_1 L_1}$ \\ \hline
    6 & $2^3 = \blue{F_3 L_3}$ \\ \hline
    10 & $5 \cdot 11 = \red{F_5 L_5}$  \\ \hline
    14 & $13 \cdot 29 =\red{F_7 L_7}$ \\ \hline
    18 & $2^3 \cdot 17 \cdot 19 = \blue{F_9 L_9}$ \\ \hline
    22& $89 \cdot 199 = \red{F_{11} L_{11}}$\\\hline
    \end{tabular}
    \hspace{.1in}
    \begin{tabular}{|c||c|}
    \hline
    $k$ & $\FibSumSquared$\\ \hline\hline
    3 & 2 \\ \hline
    7 & 1 \\ \hline
    11 & 1 \\ \hline
    15 & 2 \\ \hline
    19 & 1 \\ \hline
    23 & 1\\ \hline
    \end{tabular}
\end{center}
The four tables above partition the possible $k$-values into residue classes modulo 4. Observe that in the third table, namely when $k \equiv 2 \pmod{4}$, we highlight in red the fact that $\FibSumSquared$ values factor into two distinct primes, namely $F_{k/2}$ and $L_{k/2}$. However it is well known that $F_n L_n = F_{2n}$. Hence, for these aforementioned $k$-values, we conjecture that $\FibSumSquared = F_{k}$. Looking closer at the conjectural $\FibSumSquared$ value when $k=18$, observe that $F_9 = 2 \cdot 17$ and $L_9 = 2^2 \cdot 19$, and their product is indeed the conjectured $\FibSumSquared$ value $2^3 \cdot 17 \cdot 19$. This occurs also for all the values in the table above for $k \equiv 0 \pmod{4}$, so the phenomena of $\FibSumSquared = F_{k}$ does seems to hold for all even $k$ values.

Further computational evidence does support the conjecture that $\FibSumSquared = F_k$ for all even $k$ values. We feel this is simply too beautiful a result to not be true. Of course, the ultimate goal would be to prove this result and extend it to the Lucas setting to find $\LucSumSquared$ and more generally $\GibSumSquared$ for any Gibonacci sequence.
\end{question}


\section{Acknowledgments}
We thank the University of Wisconsin-Eau Claire's mathematics department where much of this research was conducted. We also thank Julianna Tymoczko for her advice that helped improve the exposition in this paper. Additionally, we thank Marc Renault for helpful correspondences over email. We are also grateful for the program \texttt{Mathematica}, which helped us produce much of the raw data leading to many of our conjectures that eventually became main results in this paper. Lastly, we greatly appreciate the thorough reading given by the anonymous referee who helped us streamline some of our proofs with clearer and more succinct arguments.


\bigskip
\hrule
\bigskip

\noindent 2010 {\it Mathematics Subject Classification}:
Primary 11B39, Secondary 11A05, 11B50.

\noindent \emph{Keywords: }
Fibonacci sequence, Lucas sequence, greatest common divisor, Pisano period

\bigskip
\hrule
\bigskip

\noindent (Concerned with sequences
\seqnum{A210209} and
\seqnum{A229339}.)

\end{document}